\documentclass[11pt]{amsart}
\usepackage{amsmath,amssymb,amsfonts,url}
\usepackage[usenames,dvipsnames]{pstricks}
\usepackage{epsfig}
\usepackage{pst-grad} 
\usepackage{pst-plot} 
\numberwithin{equation}{section}

\newtheorem{thm}{Theorem}[section]
\newtheorem{lem}{Lemma}[section]

\newtheorem{prop}{Proposition}[section]
\newtheorem{cor}{Corollary}[section]

\begin{document}
\title[regularity estimates for quasilinear equations]{Optimal $C^{1,\alpha}$ estimates for a class of elliptic quasilinear equations} 
\author{Dami\~ao J. Ara\'ujo}
\address{Universidade Federal da Para\'iba, Department of Mathematics, Jo\~ao Pessoa 58.051-900, Brazil} \email{araujo@mat.ufpb.br}

\author{Lei Zhang}
\address{Department of Mathematics\\
        University of Florida\\
        358 Little Hall P.O.Box 118105\\
        Gainesville FL 32611-8105}
\email{leizhang@ufl.edu}

\date{\today}

\begin{abstract} In this article we establish sharp $C^{1,\alpha}$ estimates for weak solutions of singular and degenerate quasilinear elliptic equation 
$$-\,div\, a(x, \nabla u) = f,$$
 which includes the standard $p$-laplacean equation with varying coefficients as a special case. The sharp exponent $\alpha$ is asymptotically optimal and is determined by the H\"older regularity of the coefficients, the exponent $p$ and the $q$-integrability of the source term $f$.
\end{abstract}

\maketitle

\section{Introduction}

The main goal of this article is to investigate sharp local regularity for solutions of singular and degenerate elliptic equations with varying coefficients
\begin{equation}\label{prop10}
- \,div\, a(x, \nabla u) = f(x) \quad \mbox{ in }\quad \Omega,
\end{equation}
for a bounded domain $\Omega \subset \mathbb R^n$ and dimension $n\geq 2$. The vector field $a=a(x,\xi): \Omega \times \mathbb{R}^n \to \mathbb{R}^n$ is $C^1$-regular in the gradient variable $\xi$, and satisfies the following structural assumptions: for each $x,x_i\in \Omega$ and $\xi,\xi_i\in \mathbb R^n$, $i=1,2$, there holds
\begin{equation}\label{a-cond}
\left\{
\begin{array}{c}
|a(x,\xi)|+|\partial_{\xi} a(x,\xi)||\xi| \leq \Lambda|\xi|^{p-1} \\[0.1cm]
\lambda |\xi_1|^{p-2}|\xi_2|^{2} \leq\; \big \langle\partial_{\xi} a(x,\xi_1)\xi_2, \xi_2 \big \rangle \\[0.1cm]
|a(x_1,\xi)-a(x_2,\xi)| \leq \omega(|x_1-x_2|)|\xi|^{p-1},
\end{array}
\right.
\end{equation}
for $2-1/n<p$, positive constants $\lambda \leq \Lambda$ and $\omega: [0,\infty) \to [0,\infty)$ a nondecreasing function satisfying $\omega(0)=0$. Hereafter,  we
assume
\begin{equation}\label{rhs}
\omega \in C^{0,\sigma}(\Omega) \quad \mbox{and} \quad f \in L^q(\Omega)
\end{equation}
for some $0<\sigma<1$ and $n<q\le \infty$. One special model of (\ref{prop10}) is the nonhomogeneous $p$-laplacean equation with coefficients
 \begin{equation}\label{protot}
-\, div\, (\gamma(x)|\nabla u|^{p-2}\nabla u)=f(x),
\end{equation}
where $0 < \lambda \leq \gamma(\cdot) \leq \Lambda$ is a H\"older-continuous function.


The investigation of gradient regularity properties for solutions of equation (\ref{prop10}) has been a central subject of research since the fundamental work of Uraltseva \cite{U}, who established the $C^{1,\mu}$ estimates for solutions of the $p-$harmonic equation
\begin{equation}\nonumber
-\Delta_p u:=-div\,(|\nabla u|^{p-2}\nabla u)=0, \quad p>1.
\end{equation}
See also \cite{tolks,evans} and the reference therein. Since then, regularity estimates for equations with varying coefficients as in \eqref{prop10} have been established by DiBenedetto \cite{DB} and Tolksdorf \cite{tolks}.
Here we mention two papers closely related to our article. In \cite{DM,KM} Duzaar, Kussi and Mingione established modulus of continuity for $\nabla u$, if
$$a(\cdot, \xi) \quad \mbox{ is Dini-continuous and  } \quad  f \in L(n,1/(p-1)),  $$
 where $p>2-1/n$, $L(n,1/(p-1)) \supset L^n$ is a Lorentz space. These assumptions are essentially optimal for $C^1$-regularity of solutions. In particular if \eqref{rhs} holds, a modulus of continuity of $\nabla u$ is obtained. It was discovered by Teixeira \cite{T1} that, under low regularity assumptions on the coefficients of equation (\ref{prop10}), solutions are still surprisingly smooth around the critical points of $u$. This is unexpected because at these points the equation is not uniformly elliptic: the coefficients tend to either infinity ($1<p<2$) or zero ($p>2$).


Based on regularity of the coefficients of \eqref{prop10} and the integrability of the source $f$ described in \eqref{rhs}, our main purpose is to establish the sharp index for the modulus of continuity of $\nabla u$.   Before stating our main results we list a few well known examples of (\ref{prop10}), which indicate the best regularity exponents in ideal situations. First of all,  direct computation shows that
$$
v(x)=C_p|x|^{\frac{p}{p-1}} \quad  \mbox{solves} \quad \Delta_p v=n,
$$
for $C_p=\frac{p-1}p$. This example tells that, in the case $\gamma\equiv 1$ and $q=\infty$, if the optimal regularity for (\ref{prop10}) is $C^{1,\alpha}$, $\alpha$ cannot be greater than $\frac{1}{p-1}$. The second example
$$
\overline{v}(x)=|x|^{1+\nu} \quad \mbox{solves} \quad \Delta_p\overline v=(1+\nu)^{p-1}\nu(p-1)|x|^{\nu (p-1)-1}.
$$
Clearly $\Delta_p\overline{v}\in L^q(B_1)$ for some $q>n$ and $\overline{v}\in C^{1+\nu}$ whenever $\nu=\frac{1-n/q}{p-1}+\epsilon$ for each $\epsilon>0$. Thus if $f$ is only assumed to be in $L^q$ for some $q>n$, one cannot expect $\alpha$ to be greater than $\frac{1-n/q}{p-1}$. The third reasonable obstruction comes from  the $\sigma$-H\"older continuity of $\omega$, which implies $\alpha\le \sigma$, because even for the linear homogeneous elliptic case
$$
-div(a_{ij}(x)\nabla u)=0,
$$
is necessary to assume $a_{ij}\in C^{0,\sigma}$ in order to obtain local $C^{1,\sigma}$ estimates.

Based on these obstructions we define the following quantity:
\begin{equation}\label{scaling}
\alpha_{\sigma,p,q}:= \left\{\begin{array}{rl}
\displaystyle\min\left\{\sigma,1-\frac{n}{q}\right\}\cdot \min\left\{1,\frac{1}{p-1}\right\} \, & \mbox{ if} \quad n<q<\infty,\\[0.3cm]
\sigma \cdot \displaystyle\min\left\{1,\frac{1}{p-1}\right\} \, & \mbox{ if} \quad q=\infty.\\
\end{array}\right.
\end{equation}

For the constant coefficient field $a(x,\xi)=a(\xi)$, case $\omega\equiv 0$, it is well known that weak solutions of
\begin{equation}\label{consteq}
-div\,(a(\nabla u))=0
\end{equation}
are locally $C^{1,\alpha_m}$ for a maximal $\alpha_m\in (0,1)$ depending only on $n, p ,\lambda$ and $\Lambda$. By the aforementioned examples we cannot expect the optimal exponent $\alpha$ greater than $\min\{\alpha_m,\alpha_{\sigma,p,q}\}$. Our main result provides the following asymptotic optimal regularity estimate.

\begin{thm}\label{thm1} Let $u \in W^{1,p}(\Omega)$ be a weak solution of \eqref{prop10} under the assumptions \eqref{a-cond} and \eqref{rhs}. Then $u$ is locally $C^{1,\alpha}(\Omega)$, where
\begin{equation}\label{opt-gamma}
\alpha = \min\{\alpha_m^-,\alpha_{\sigma, p,q}\}.
\end{equation}
Moreover, for any $K\Subset \Omega$,
there holds
\begin{equation}\label{mainest}
\sup_{x,y \, \in \,K, \,x\neq y}\frac{|\nabla u(x)-\nabla u(y)|}{|x-y|^{\alpha}}\le C,
\end{equation}
for some $C>0$ depending only on $\alpha, n,p,\Lambda,\lambda,\|u\|_{W^{1,p}},\|\omega\|_{C^{0,\sigma}},\|f\|_{L^q}$ and $dist(K,\partial\Omega)$.
\end{thm}

Hereafter in this paper we denote $\alpha$ as in \eqref{opt-gamma} with $\alpha_m^-$ denoting any positive number less than $\alpha_m$. So in this case, the constant $C=C(\alpha)$ may blow-up as $\alpha\to \alpha_m$.


In particular, Theorem \ref{thm1} provides an important information on how the $C^{1,\alpha}$-regularity deteriorates as q approaches $n$. It is well known, see \cite{M1,M}, that a solution $u$ of (\ref{protot}) satisfies $\nabla u\in C^{0,\alpha(\varepsilon)}$ provided $f\in L^{n+\varepsilon}$ and $\gamma\in C^{0,\varepsilon}$,
 where $\alpha(\varepsilon)\to 0$ as $\varepsilon\to 0$. However, no explicit expression of $\alpha(\varepsilon)$ is given. Using Theorem \ref{thm1} we can make $\alpha(\varepsilon)$ explicit.

\begin{cor}\label{cor1}
Let $u$ be a weak solution to \eqref{prop10} satisfying \eqref{a-cond} and \eqref{rhs} for $p \geq 2$. There exists a small number $\overline{\varepsilon}$, depending on parameters as in Theorem \ref{thm1} such that if
$$
\omega \in C^{\,0,\varepsilon} \quad \mbox{and} \quad f \in L^{n+\varepsilon}
$$
for $0<\varepsilon \leq \overline{\varepsilon}$, then $\nabla u$ is locally of class $C^{0,\alpha(\varepsilon)}$ for
$$
\alpha(\varepsilon)=\frac{\varepsilon}{n+\varepsilon}\cdot \frac{1}{p-1}.
$$
\end{cor}
\medskip

The proof of Corollary \ref{cor1} follows by choosing $\varepsilon$ sufficiently small such that $\varepsilon/(n+\varepsilon) < \alpha_m$. Therefore, it easy to check that for $p \geq 2$ we have $\alpha_{\sigma,p,q}=\frac{\varepsilon}{n+\varepsilon}\cdot \frac{1}{p-1}$ where $\alpha_{\sigma,p,q}=\alpha$, for $\alpha$ as in Theorem \ref{thm1}.


In the plane, using tools from complex analysis in \cite{LL}, Lindgren and Lindqvist obtained the optimal interior $C^{1,\alpha}$ estimate for the inhomogeneous $p$-Laplacian equations,
$$
-\Delta_p u = f \in L^q, \quad
$$
where $p\ge 2$ and $q < \infty$. For the case $f \in L^\infty$, the first author, Teixeira and Urbano \cite{AT} proved that solutions are locally $C^{1,p'-1}$ where $p'$ is the conjugated index of $p$. Still in dimension $n=2$, case $f\equiv0$, $p$-harmonic functions are locally $C^{1,\alpha_p}$ for an exponent $\alpha_p$ satisfying the following strict inequality
\begin{equation}\nonumber
\frac{1}{p-1} < \alpha_p,
\end{equation}
for any $2 \leq p < \infty$,  see \cite{AT,BK}.   Consequently, we observe that the maximal regularity exponent for the constant coefficients equation \eqref{consteq} satisfies $\alpha_m > 1/(p-1)$ when $n=2$. Thus we can apply Theorem \ref{thm1} to obtain the optimal regularity for solutions of \eqref{prop10} in two dimension spaces:

\begin{cor}[Optimal regularity in the plane] Under conditions \eqref{a-cond} and \eqref{rhs} for $2<q\leq \infty$, $p \geq 2$ and $n=2$, solutions of equation \eqref{prop10} are locally $C^{1,\varrho}$ for the optimal exponent
\begin{equation}\nonumber
\varrho=
\left\{
\begin{array}{ccc}
\dfrac{1}{p-1} \cdot
\displaystyle\min\left\{\sigma,\frac{q-2}{q}\right\} & \mbox{if}& 2 < q < \infty \\
\dfrac{1}{p-1} \cdot
\sigma & \mbox{if}& q = \infty \\

\end{array}
\right.
\end{equation}
\end{cor}


Here we mention the essential ingredients in the proof of the main results. For a generic solution of \eqref{prop10} we consider an appropriate neighbourhood of the critical set $\mathcal{C}(u):=\{x : \nabla u (x) = 0\}$ in small balls with large radii: $|\nabla u| \lesssim r$. In this case we apply the oscillation estimates developed by the first author, Teixeira and Urbano, see \cite{AT,ATU}. However, for balls of small radii: $r \lesssim |\nabla u|$ the gradient becomes large and the vector field $a(x,\xi)$ has a linear growth at infinity. In this case, Theorem \ref{thm_2} (to be established in the Section \ref{ellipsec}) essentially provides the optimal regularity estimates for such regions. By combining estimates in these situations carefully we obtain the desired local estimates.

\subsection*{Notations}\label{not} We use $B_r(x) \subset \mathbb{R}^n$ to denote the open ball with radius $r>0$ centered at $x \in \mathbb{R}^n$. If $x$ is the origin we use $B_r$ instead of $B_r(0)$.  Given a compact set $K \Subset \Omega$, we denote $\mbox{dist}(K,\partial\Omega)$ the euclidean distance between $K$ and the boundary of the domain $\Omega$.

\subsection*{Organization of the paper}\label{org}
In Section \ref{sec2}, we derive regularity estimates for balls with large radii. In Section \ref{sec3},  using optimal regularity estimates for equations with linear growth, Theorem \ref{thm_2}, we establish estimates for balls with small radii, and in Section \ref{mainthmsec}, standard arguments are employed to prove Theorem \ref{thm1}.
For the sake of clarity and completeness in the presentation of the arguments, we leave the proof of Theorem \ref{thm_2} in the Appendix, Section \ref{ellipsec}.

\section{Sharp growth estimates for large radii}\label{sec2}

First we recall that equation \eqref{prop10} satisfies \eqref{a-cond} and \eqref{rhs}. By the aforementioned works of Duzaar-Mingione \cite[theorem 4]{DM} and Kuusi-Mingione \cite[theorem 1.6]{KM}, all solutions $u$ of (\ref{prop10}) are at least $C^1$. Thus $\nabla u$ is defined at each point.

The goal in this section is to derive the following estimates for weak solutions $u$ of \eqref{prop10}:
\begin{equation}\label{estfinalsmall}
\sup\limits_{B_\rho (x_0)}|u(x)-u(x_0)-\nabla u (x_0)\cdot (x-x_0)| \le C\rho^{1+\alpha}
\end{equation}
for radii satisfying
$$
|\nabla u(x_0)| \le c\rho^{\alpha}
$$
with $c,C$ depending on $n,p,\Lambda,\lambda, \alpha,\|u\|_{W^{1,p}}, \|\omega\|_{C^{0,\sigma}}, \|f\|_{L^q}$ and $dist(K,\partial \Omega)$. Hereafter in this paper, constants with this dependence shall be called universal. Here is the main result of this section.

\begin{prop} \label{small}
Let $u$ be a weak solution of \eqref{prop10} in $\Omega$. Given a compact $K \Subset \Omega$ and $x_0 \in K$, there exist universal positive constants $\kappa,C$ and $\overline{\rho}$, such that if
\begin{equation}\label{E7}
|\nabla u(x_0)| \leq \kappa \rho^{\,\alpha},
\end{equation}
for some $0<\rho\leq\overline{\rho}$, then
\begin{equation}\label{smallestim}
\sup\limits_{B_{\rho}(x_0)}|u(x)-u(x_0)| \leq  C\rho^{1+\alpha}.
\end{equation}
\end{prop}

In order to derive the Proposition \ref{small} we show, under a certain smallness assumption of certain parameters, $u$ can be approximated by a solution of a constant coefficient equation. For the sake of clarity, we restrict our analysis to the simplest case $\Omega=B_1$ and $x_0=0$.

\begin{lem}\label{propcomp}
Given $\epsilon>0$ there exists $\delta>0$ depending on $n$,$p$,$q$,$\Lambda$,$\lambda$,$\sigma$ and $\epsilon$, such that if 
\begin{equation}\label{jampa}
\|f\|_{L^q(B_1)} \leq \delta,
\quad \sup\limits_{B_1} |a(x,\xi)-a(0,\xi)| \leq \delta\, |\xi|^{p-1}
\end{equation}
and \eqref{a-cond} holds, then for each weak solution $u$ of \eqref{prop10} in $B_1$ satisfying $\|u\|_{L^\infty(B_1)} \leq 1$ and $u(0)=0$, there exists a function $h$ in $B_{3/4}$, weak solution of
\begin{equation}\label{homeq}
-\mbox{div}\,(\overline{a}(\nabla h))=0 \quad \mbox{in} \quad B_{3/4} \quad \mbox{with} \quad h(0)=0,
\end{equation}
for a constant coefficients field $\overline{a}$, satisfying \eqref{a-cond} with $\omega \equiv 0$, such that
\begin{equation}\label{comp-estimate}
\sup\limits_{B_{1/2}}|u-h| + |\nabla u(0)-\nabla h (0)| \leq \epsilon.
\end{equation}
\end{lem}

\begin{proof}
By way of contradiction we assume that there exist $\epsilon_\star>0$ for which the Lemma fails. This means that we can find sequences $\{u_j\}$, $\{a_j\}$, $\{f_j\}$ and $\{\delta_j\}$, for $j\in \mathbb N$ satisfying
\begin{equation}\label{E1}
\mbox{div} (a_j(x,\nabla u_j)) = f_j \quad \mbox{in} \quad B_1
\end{equation}
where $\|u_j\|_{L^\infty(B_1)} \leq 1$ and $u_j(0)=0$ as well as
\begin{equation}\label{E2}
\|f_j\|_{L^q(B_1)} \leq \delta_j \;\; \mbox{and} \;\;  |a_j(x,\xi)-a_j(0,\xi)|\cdot |\xi|^{1-p} \leq \delta_j, \;\; \mbox{for} \;\; \delta_j \to 0
\end{equation}
as $j \to \infty$. However, for any solution $h$ of a homogeneous constant coefficients equation, as in \eqref{homeq}, satisfying $h(0)=0$, there holds
\begin{equation}\label{E3}
\sup_{B_{1/2}}|u_j-h| + |\nabla h(0)-\nabla u_j (0)| > \epsilon_\star
\end{equation}
for some positive parameter $\epsilon_\star$.

In view of this, a standard regularity result for solutions of (\ref{E1}) assures that $\{u_j\}$
is a pre-compact sequence in the $C^1$-topology, see \cite{DM,KM} and \cite[Theorem 2.1]{T1}. Therefore, along a subsequence, $\{u_j\}$ converges locally to a function $u_{\infty}$ in $C^1$ and the following estimate holds:
\begin{equation}\label{contradiction}
u_\infty(0)=0 \quad \mbox{and} \quad \sup\limits_{B_{1/2}}|u_j-u_\infty|+ |\nabla u_j(0)-\nabla u_\infty(0)|\to 0.
\end{equation}
Moreover, thanks to the $C^1$-compactness of $u_j$, there exists a universal constant $L>0$ such that $|\nabla u_j| \leq L/2$ in $B_{3/4}$. Now, let us define
$$
b_j(x,\xi):=a_j(x,\xi)\chi_{\{|\xi|\leq L\}}+K\chi_{\{|\xi|> L\}}.
$$
Clearly the sequence $\{b_j(0,\cdot)\}$ is bounded and equicontinuous, therefore by the well known theorem of Ascoli-Arzel\'a, $b_j(0,\cdot) \to b_\infty(0,\cdot)$ uniformly in $B_{1/2}$. Hence, by \eqref{E2} we obtain
\begin{equation}\label{ufpb}
|a_j(x,\xi)-b_{\infty}(0,\xi)| \leq L^{p-1}\delta_j+|a_j(0,\xi)- b_\infty(0,\xi)|
\end{equation}
for any $x \in B_{_{1/2}}$ and $\xi \in B_{L}$. This means that $a_j \to b_{\infty}$ uniformly in $B_{1/2} \times B_{L}$. From \eqref{contradiction} and \eqref{ufpb}, we have by standard arguments that $u_\infty$ solves the constant coefficients equation
\begin{equation}\nonumber
-\mbox{div} (b_\infty(0,\nabla u_\infty))=0 \quad \mbox{in} \quad B_{3/4} \quad \mbox{with} \quad u_\infty(0)=0.
\end{equation}
Notice that in $B_{3/4}$, $u_\infty$ works as a function $h$ described in \eqref{E3}. Therefore,
$$
0<\epsilon_\star \leq \sup\limits_{B_{1/2}}|u_j-u_\infty|+|\nabla u_j(0)-\nabla u_\infty(0)|.
$$
This leads a contradiction to \eqref{contradiction} for $j \gg 1$.
\end{proof}

\medskip

\begin{lem}\label{stepone}
Let $u$ be a weak solution of \eqref{prop10} in $B_1$ with $\|u\|_{L^\infty(B_1)}\leq 1$ and $u(0)=0$. There exist small positive constants $\delta_0$ and $\rho_0$ depending only on 
$n$,$p$,$q$,$\Lambda$,$\lambda$,$\sigma$ and $\alpha$ such that if 
\begin{equation}\label{pb}
\|f\|_{L^q(B_1)} \leq \delta_0, \quad \sup\limits_{B_1} |a(x,\xi)-a(0,\xi)| \leq \delta_0 |\xi|^{p-1}
\end{equation}
and
$$
|\nabla u(0)| \leq \frac{1}{4}\,\rho_0^{\,\alpha},
$$
there holds
$$
\sup\limits_{B_{\rho_0}}|u(x)| \leq  \rho_0^{1+\alpha}.
$$
\end{lem}

\begin{proof}
For $\epsilon>0$ to be determined later, we can find a solution $h$ of some constant coefficient equation satisfying \eqref{jampa} for $\delta=\delta(\epsilon)$, such that
\begin{equation}\label{2-e1}
|\nabla u(0)-\nabla h(0)| \leq \epsilon \quad \mbox{and} \quad  \sup\limits_{B_{\rho}}|u| \leq  \epsilon +\sup\limits_{B_{\rho}}|h|
\end{equation}
for any $0<\rho\le 1/2$. On the other hand, by the local regularity estimates to constant coefficients equations together the fact $h(0)=0$, we get
\begin{equation}\label{2-e2}
\sup\limits_{B_{\rho}}|h| \leq C \rho^{1+\alpha_m} + |\nabla h(0)|\rho.
\end{equation}
Here, we emphasize that $C$ depends only on $n$ and $p$.
As a consequence of (\ref{2-e1}) and (\ref{2-e2}) we obtain
\begin{equation}\label{E5}
\sup\limits_{B_{\rho}}|u| \leq \epsilon + C \rho^{1+\alpha_m} + (\epsilon+|\nabla u(0)|)\,\rho.
\end{equation}
Setting
$$
\rho=\rho_0 := \left(\frac{1}{4C}\right)^{\frac{1}{\alpha_m-\alpha}} \quad \mbox{and} \quad  \epsilon =\frac{1}{4} \rho_0^{1+\alpha},
$$
we get by \eqref{E5},
\begin{equation}\nonumber
\sup\limits_{B_{\rho_0}}|u| \leq \rho_0^{1+\alpha}
\end{equation}
after a universal choice of $\delta=\delta(\rho_0)$.
\end{proof}

Next, we apply Lemma \ref{stepone} iteratively to prove a special case of Proposition \ref{small}:

\begin{lem} \label{Interaction} Let $u$ be a weak solution of \eqref{prop10} in $B_1$ with $\|u\|_{L^\infty(B_1)} \leq 1$ and $u(0)=0$. Under conditions \eqref{a-cond} and \eqref{pb} and the same assumption for $\rho_0$ as in Lemma \ref{stepone}, there exists $C>0$ depending only on 
$n$,$p$,$q$,$\Lambda$,$\lambda$,$\sigma$ and $\alpha$ such that if
$$
|\nabla u(0)| \leq \frac{1}{4}\rho^{\alpha},
$$
for some $0<\rho \leq \rho_0$, then
$$
\sup\limits_{B_{\rho}}|u(x)| \leq  C\rho^{1+\alpha}.
$$
\end{lem}

\begin{proof}
First, we show a discrete version of Lemma \ref{Interaction}. More precisely, we claim that for $\rho_0$ given as in Lemma \ref{stepone}, if
$$
|\nabla u(0)| \leq \frac{1}{4}\rho_0^{k\alpha}
$$
for some positive integer $k$, then
$$
\sup\limits_{B_{\rho_0^k}}|u(x)| \leq  \rho_0^{k(1+\alpha)}.
$$
We prove it inductively. The case $k=1$ follows by Lemma \ref{stepone}. Next, we assume the conclusion holds for $k\le i$. Suppose $u$ satisfies
$$
|\nabla u(0)| \leq \frac{1}{4}\rho_0^{(i+1)\alpha}.
$$
Set $u_i(x):=\dfrac{u(\rho_0^ix)}{\rho_0^{i(\alpha+1)}}$. Direct computation shows that $u_i$ solves
$$
-div\; a_i(x,\nabla u_i(x))=f_i(x) \quad \mbox{ in }\quad B_1
$$
for
$$
a_i(x,\xi):= \rho_0^{i\,\alpha(1-p)}a(\rho_0^i\, x, \rho_0^{i\,\alpha}\, \xi)
$$
which satisfies the same conditions as in \eqref{a-cond}, and also the smallness conditions as in \eqref{pb}.  $f_i(x)$ is defined by
$$
f_{i}(x):=\rho_0^{i(\alpha(1-p)+1)}f(\rho_0^i x).
$$
It is easy to verify that
$$
\|f_{i}\|_{L^q(B_1)}\le  \rho_0^{i(\alpha(1-p)+1-\frac nq)} \|f\|_{L^q(\Omega)}.
$$
Note that, since $\alpha\le (1-n/q)/(p-1)$ for $p\geq 2$, we have the exponent $\alpha(1-p)+1- n/q \geq 0$ for any $p \geq 2-1/n$ and so we obtain $\|f_i\|_{L^q(B_1)}\le \delta_0$.  In addition we get
$$
|\nabla u_i(0)| \leq \frac{1}{4}\rho_0^{\,\alpha}.
$$
Thus, $u_i$ satisfies the hypotheses of Lemma \ref{stepone} and we have
$$
\sup\limits_{B_{\rho_0}}|u_i(x)| \leq  \rho_0^{1+\alpha}
$$
Finally, by the definition of $u_i$ and the estimate above, we obtain
$$
\sup\limits_{B_{\rho_0^{i+1}}}|u| \leq \rho_0^{(i+1)(1+\alpha)}.
$$

To conclude the proof of Lemma \ref{Interaction} we proceed as follows. Given a number $0<\rho \leq \rho_0$, we select an integer $k>0$ such that
$$
\rho_0^{k+1}<\rho \leq \rho_0^{k}.
$$
Hence, by the condition \eqref{E7} we have in particular
$$
|\nabla u(0)| \leq \frac{1}{4}\rho_0^{\,k\alpha}
$$
therefore,
\begin{equation}\label{E8}
\sup\limits_{B_{\rho}}|u(x)| \leq \sup\limits_{B_{\rho_0^k}}|u(x)| \leq  \rho_0^{k(1+\alpha)} = \rho_0^{-(1+\alpha)} \rho^{1+\alpha}
\end{equation}
and Lemma \ref{Interaction} is established.
\end{proof}

\subsection*{Proof of Proposition \ref{small}}
The idea here is to scale a fixed weak solution $u$ of \eqref{prop10} and apply Lemma \ref{Interaction}.

In fact, for a fixed compact set $K \Subset \Omega$ and $x_0 \in K$, we denote $d_0=dist(K,\partial\Omega)$. Also, for positive parameters $A_0$ and $B_0$, we define the following function:
$$
v(x):= \frac{u(x_0+ A_0 \, x)-u(x_0)}{B_0}.
$$
Note that $v$ solves
$$
-\mbox{div}\, a_0(x,\nabla v) = f_0 \quad \mbox{in} \quad B_1,
$$
for
$$
a_0(x, \xi):= \left( B_0/A_0 \right)^{1-p} a(x_0+A_0\, x,B_0/A_0 \cdot \xi)
$$
and
$$
\|f_0\|_{L^q(B_1)} \leq B_0^{1-p} A_0^{p-n/q} \|f\|_{L^q(B_1)}.
$$
Also, it is easy to see that $a_0$ satisfies the structural condition \eqref{a-cond} with
\begin{equation}\nonumber
|a_0(x, \xi)-a_0(0, \xi)| \leq \omega(A_0\,|x|).
\end{equation}
Therefore, by choosing
$$
A_0 := \min\left\{1,d_0/2 , \omega^{-1}\left(
\delta_0  \right)/ d_0 \right\}
$$
and
$$
B_0 := \max \left\{1,2\|v\|_{L^\infty(\Omega)},\sqrt[p-1]{\|f\|_{L^q(\Omega)} \delta_0^{-1} }\right\},
$$
we see that $a_0$ satisfies the same structural conditions \eqref{a-cond} as well as the smallness assumptions \eqref{pb} for $\delta_0$. Moreover $v$ satisfies $v(0)=0$ and $\|v\|_{L^\infty(B_1)} \leq 1$.

\medskip

Finally, let us conclude the proof of Proposition \ref{small}. For every radius $0< \rho \leq d_0\, \rho_0$ and $c>0$ to be chosen later, we have that
$$
|\nabla u(x_0)| \leq c\rho^\alpha
\quad
\mbox{implies}
\quad
|\nabla v(0)| \leq c\,\frac{B_0d_0^{\,\alpha}}{A_0}\tilde \rho^{\,\alpha}
$$
for $\tilde\rho = \rho / d_0 \leq \rho_0$. By selecting $c:=A_0/(4B_0d_0^\alpha)$ we are able to apply Lemma \ref{Interaction} where
$$
\sup_{B_{\tilde\rho}}|v(x)| \leq C \tilde\rho^{1+\alpha}
$$
and so,
$$
\sup_{B_{\frac{A_0}{d_0}\rho}(x_0)}|u(x)-u(x_0)| \leq \frac{C}{A_0^\alpha} \left(\frac{A_0}{d_0}\rho\right)^{1+\alpha}.
$$
Therefore, we conclude that for each $0<r \leq \rho_0d_0/2$, if
$$
|\nabla u(x_0)| \leq \overline{c}\, r^\alpha
$$
where $\overline{c}=A_0^{1-\alpha}/(4B_0)$, there holds
$$
\sup_{B_{r}(x_0)}|u(x)-u(x_0)| \leq \overline{C} r^{1+\alpha}
$$
for some universal $\overline{C}>0$. The proof of Proposition \ref{small} is complete.

\medskip

Also, we would like to point out that by regularity theory for quasilinear equations  the local upper bound for $L^\infty$-norm follows: $\|u\|_{L^{\infty}(K)} \leq C\|u\|_{L^{p}(\Omega)}$, for any compact set $K \Subset\Omega$ and a universal constant $C>0$ depending on universal parameters, specially on  $dist(K,\partial\Omega)^{-1}$. Therefore, the normalization constant $B_0$ depends only on the $L^p$-norm of $u$ and universal parameters.

\section{Optimal regularity estimates for small radii}\label{sec3}

Here we shall derive regularity estimates for radii satisfying
$$
\kappa \rho^\alpha < |\nabla u(x_0)|
$$
for $\kappa>0$ universal as in Proposition \ref{small}. In this case a crucial observation is that the equation \eqref{prop10} behaves as a quasilinear equation with linear growth, i.e., it satisfies conditions \eqref{a-cond} and \eqref{rhs} for $p=2$. In this special case we shall use the following result,  Theorem \ref{thm_2}, to assert that solutions of \eqref{prop10} are $C^{1+\beta}$-regular for a given exponent $\beta=\beta(\sigma,q) \geq \alpha_{\sigma, p, q}$ for all $p>1$.

\begin{thm}\label{thm_2} Let $u \in H^{1}(\Omega)$ be a solution to \eqref{prop10} satisfying the conditions of Theorem \ref{thm1} with p=2. Then $u$ is locally $C^{1,\beta}$, where the estimate \eqref{mainest} is valid for the following exponent:
$$
\beta=\left\{
\begin{array}{cl}
\min\left\{\sigma, 1-n/q\right\} & \mbox{ if} \quad n<q<\infty  \\
\sigma  & \mbox{ if} \quad q=\infty.
\end{array}
\right.
$$
\end{thm}

Theorem \ref{thm_2} and Proposition \ref{small} are the key ingredients for proving the $C^{1,\alpha}$-regularity estimates stated in Theorem \ref{thm1}.
Even for the case $p=2$, the nonlinear vector field $a$ satisfying
\eqref{a-cond} and \eqref{rhs} has a linear growth and the classical regularity estimates for elliptic equations of divergent form cannot be applied directly. In order not to interrupt the proof the main theorem, we defer the proof of Theorem \ref{thm_2} to the appendix in Section \ref{ellipsec}. We recall that constants are called universal if they only depend on $n$,$p$,$\Lambda$,$\lambda$,$\alpha$, $\|u\|_{W^{1,p}}$,$\|\omega\|_{C^{0,\sigma}}$,$\|f\|_{L^q}$ and $dist(K,\partial \Omega)$.

\begin{prop}\label{large}
Let $u \in W^{1,p}(\Omega)$ be a weak solution of \eqref{prop10} in $\Omega$ satisfying \eqref{a-cond} and \eqref{rhs} and a compact $K \Subset \Omega$. For $\kappa$ as in Proposition \ref{small}, there exist universal positive constants $\tilde\rho$ and $C$, such that if
\begin{equation}\label{smallradius}
\kappa\rho^{\alpha} < |\nabla u(x_0)|
\end{equation}
for $x_0 \in K$ and $0<\rho\leq \tilde\rho$ , then
\begin{equation}\label{largeestim}
\sup\limits_{B_{\rho}(x_0)}|u(x)-u(x_0)-\nabla u(x_0)\cdot (x-x_0)| \leq  C\rho^{1+\alpha}.
\end{equation}
\end{prop}

\begin{proof}
Initially, for $\rho_\star:= [\kappa^{-1} \cdot |\nabla u(x_0)|]^{\frac{1}{\alpha}}$ let us define the rescaled function
$$
v(x):=\frac{u(\rho_\star x+x_0)-u(x_0)}{\rho_\star^{1+\alpha}} \quad \mbox{in} \quad B_1.
$$
Direct computation gives
\begin{equation}\label{finaleq}
-div\; a_\star(x, \nabla v)=f_\star(x) \quad \mbox{in} \quad B_1,
\end{equation}
where
$$
a_\star(x,\xi)=\rho_\star^{\alpha(1-p)}a(\rho_\star x,\rho_\star^{\alpha} \xi) \quad \mbox{and} \quad f_\star(x):=\rho_\star^{\alpha(1-p)+1}f(\rho_\star x).
$$
Clearly $a_\star$ satisfies the conditions in \eqref{a-cond} and \eqref{rhs}, where in particular
$$
\|f_\star\|_{L^q(B_1)}\leq \rho_\star^{\alpha(1-p)+1-n/q}\|f\|_{L^q(\Omega)} \leq \|f\|_{L^q(\Omega)}.
$$

In the case $\rho_\star \leq \overline{\rho}$, for $\overline\rho$ as in Proposition \ref{small}, we can apply estimate \eqref{smallestim} precisely for the radius $\rho=\rho_\star$. So, we find a universal $C>0$ such that
$$
\|v\|_{L^p(B_1)} \leq \sup\limits_{x \in B_1}|v(x)| \cdot |B_1|^{1/p} = \sup\limits_{x \in B_{\rho_\star}} \frac{|u(x)-u(x_0)|}{\rho_\star^{\alpha}}\cdot |B_1|^{1/p} \leq C.
$$
Consequently, by applying $C^0$-estimates to $\nabla v$, there exists $\tau_\star>0$ such that
$$
\mbox{osc}_{B_{\tau_\star}}|\nabla v|<\frac{\kappa}{2}.
$$
Since $|\nabla v(0)|=\kappa$, we have $|\nabla v(x)|>\frac{\kappa}{2}$ in $B_{\tau_\star}$. From this, we can find an universal constant $c_0>0$, such that
\begin{equation}\label{controlgrad}
c_0 \leq |\nabla v(x)| \leq c_0^{-1} \quad \mbox{for} \quad x \in B_{\tau_\star}.
\end{equation}
Therefore, in view of \eqref{controlgrad}, the equation \eqref{finaleq} is a nonlinear partial differential equation with linear growth, i.e., with $a_\star$ satisfying the conditions
\begin{equation}\nonumber
\left\{
\begin{array}{c}
|a_\star(x,\xi)|+|\partial_{\xi} a_\star(x,\xi)||\xi| \leq c_0^{2-p}\Lambda|\xi| \\[0.1cm]
\lambda c_0^{\,p-2}|\xi_2|^{2} \leq\; \big \langle\partial_{\xi} a_\star(x,\xi_1)\xi_2, \xi_2 \big \rangle\\[0.05cm]
|a_\star(x_1,\xi)-a_\star(x_2,\xi)| \leq c_0^{2-p}\omega(|x_1-x_2|)|\xi|,
\end{array}
\right.
\end{equation}
for each $x,x_i \in B_1$ and $\xi,\xi_i \in \mathbb R^n$, $i=1,2$. In other words, $a_\star$ satisfies \eqref{a-cond} for $p=2$ within $B_{\tau_\star}$. Hence, Theorem \ref{thm_2} provides the following estimate
\begin{equation}\label{eqfinal1}
\sup\limits_{B_r} |v(x)-v(0) - \nabla v(0) \cdot x| \leq C_1 r^{\beta},
\end{equation}
for each $0< r \leq \tau_\star/2$, where
$$
\beta=\left\{\begin{array}{clc}
\min\left\{ \sigma,1-n/q \right\} & \mbox{if} & q<\infty, \\
\sigma & \mbox{if} & q=\infty.
\end{array}
\right.
$$
As commented in the beginning of this section, one is easy to see that $\beta>\alpha_{\sigma,p,q}$ for any $0<\sigma < 1$, $p>1$ and $n<q \leq \infty$.
Therefore, by \eqref{eqfinal1} we have
\begin{equation}\label{eqfinal2}
\sup\limits_{B_{\rho_\star r}(x_0)} |u(x)-u(x_0) - \nabla u(x_0)\cdot (x-x_0)| \leq C_1 \tau_\star^{1+\alpha} r^{1+\beta} \leq C_1 (\rho_\star r)^{1+\alpha}
\end{equation}
for every $0< r \leq \tau_\star/2$. Hence, the estimate \eqref{largeestim} holds whenever
$$
0< r \leq \tau_\star\rho_\star/2.
$$
To conclude this case, we have to show that the estimate \eqref{eqfinal2} also holds for
$$
\tau_\star\rho_\star/2 < r < \rho_\star.
$$
Since the estimate \eqref{smallestim} holds precisely for the radius $\rho_\star$, we have
$$
\sup\limits_{B_r(x_0)} |u(x)-u(x_0) - \nabla u(x_0) \cdot (x-x_0)| \leq C \rho_\star^{1+\alpha} \leq C \left(\frac{2}{\tau_\star}\right)^{1+\alpha} r^{1+\alpha},
$$
which concludes the case $\rho_\star \leq \overline{\rho}$.

Finally, for the case $\rho_\star > \overline\rho$, we have the following universal bound
$$
|\nabla u(x_0)|> \kappa \overline\rho^{\,\alpha}.
$$
Therefore,  the argument applied previously for the function $v$ can also be applied to $u$. Thus, there exists a small universal parameter $\rho\,'$ such that estimate \eqref{eqfinal2} holds for every $0<\rho\leq \rho'$.

We conclude the proof of Proposition \ref{large} by choosing $\tilde\rho := \min\{\overline\rho,\rho'\}$.
\end{proof}

\begin{figure}[h!]
\begin{center}
\scalebox{1} 
{
\begin{pspicture}(0,-3.0644138)(12.449376,3.0144138)
\definecolor{color10}{rgb}{0.6039215686274509,0.5882352941176471,0.5882352941176471}
\definecolor{color11}{rgb}{0.5058823529411764,0.47843137254901963,0.47843137254901963}
\definecolor{color15}{rgb}{0.4627450980392157,0.4549019607843137,0.4549019607843137}
\definecolor{color153}{rgb}{0.49019607843137253,0.47843137254901963,0.47843137254901963}
\definecolor{color1}{rgb}{0.611764705882353,0.596078431372549,0.596078431372549}
\definecolor{color5}{rgb}{0.2784313725490196,0.25882352941176473,0.25882352941176473}
\definecolor{color8}{rgb}{0.25882352941176473,0.24313725490196078,0.24313725490196078}
\definecolor{color175}{rgb}{0.3254901960784314,0.3254901960784314,0.3254901960784314}
\definecolor{color7}{rgb}{0.24313725490196078,0.22745098039215686,0.22745098039215686}
\definecolor{color14c}{rgb}{0.7294117647058823,0.7215686274509804,0.7215686274509804}
\definecolor{color15g}{rgb}{0.8431372549019608,0.8470588235294118,0.8588235294117647}
\definecolor{color15f}{rgb}{0.8470588235294118,0.8588235294117647,0.8588235294117647}
\psbezier[linewidth=0.004,fillstyle=gradient,gradlines=2000,gradbegin=color15g,gradend=color15f,gradmidpoint=1.0](1.0,-0.7254513)(1.0,-0.7055862)(1.6725619,-1.27248)(2.822193,-1.4853005)(3.9718237,-1.6981212)(5.525446,-1.08538)(6.406506,-1.145368)(7.287566,-1.2053561)(8.119678,-2.0926712)(8.769349,-1.8852212)(9.419022,-1.6777713)(9.530265,-0.6454672)(10.050892,-0.3455267)(10.571518,-0.0455862)(11.14,-0.5454121)(11.14,-0.5255862)(11.14,-0.50576025)(11.126757,-2.2455862)(11.132193,-2.2451499)(11.1376295,-2.2447135)(1.0,-2.2455862)(1.0,-2.2455862)(1.0,-2.2455862)(1.0,-0.74531645)(1.0,-0.7254513)
\psbezier[linewidth=0.1,linecolor=color1]{cc-cc}(1.0101562,-0.7255862)(4.110156,-2.7655861)(5.494854,-0.1967586)(7.690156,-1.6055862)(9.885458,-3.0144138)(9.030156,0.7144138)(11.150155,-0.5255862)
\pscircle[linewidth=0.025999999,linecolor=color10,linestyle=dashed,dash=0.16cm 0.16cm,dimen=outer](6.380156,0.3644138){2.49}
\psline[linewidth=0.01cm,linecolor=color11](6.372031,0.4544138)(5.677656,1.5944138)
\usefont{T1}{ptm}{m}{n}
\rput(1.5145313,-1.6955862){\color{color7}$\mathcal{C}(u)$}
\psline[linewidth=0.03,linecolor=color153,tbarsize=0.07055555cm 5.0]{cc-|*}(5.4576564,0.2744138)(2.517656,1.5144138)(2.517656,1.8544137)
\pscircle[linewidth=0.025999999,linecolor=color10,linestyle=dashed,dash=0.16cm 0.16cm,dimen=outer,fillstyle=crosshatch,hatchwidth=0.01,hatchangle=45.0,hatchcolor=color14c](6.3439064,0.3881638){1.39375}
\psline[linewidth=0.01cm,linecolor=color11](8.717656,-0.4655862)(6.397656,0.4144138)
\usefont{T1}{ptm}{m}{n}
\rput(2.5081253,2.0644138){\small \color{color175}regular zone}
\usefont{T1}{ptm}{m}{n}
\rput(6.4445314,1.0244138){\color{color175}$\rho_\star$}
\usefont{T1}{ptm}{m}{n}
\rput(6.1745315,0.2644138){\color{color5}$x_0$}
\psdots[dotsize=0.09](6.3901553,0.4544138)
\psdots[dotsize=0.14,linecolor=color8](6.372031,0.4544138)
\usefont{T1}{ptm}{m}{n}
\rput(7.8845315,-0.4755862){\color{color8}$\rho$}
\usefont{T1}{ptm}{m}{n}
\rput(9.9175005,2.559414){\footnotesize \color{color175}$\rho_\star \sim |\nabla u(x_0)|^{\frac1\alpha}$}
\psframe[linewidth=0.03,dimen=outer](11.150155,3.0144138)(0.9901561,-2.2655861)
\end{pspicture}
}
\end{center}

\caption{This picture indicates how estimate \eqref{estfinal} is obtained by considering the cases: the regular case $\rho < \rho_\star$ and the degenerate one $\rho_\star \leq \rho$.}
\end{figure}
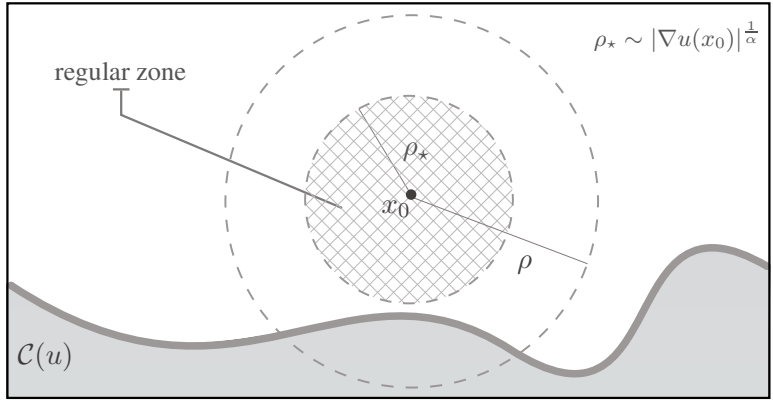

\section{Proof of main Theorem}\label{mainthmsec}

Finally we combine Proposition \ref{small} and Proposition \ref{large}, which are for the large and small radii cases respectively, in the proof of Theorem \ref{thm1}, see figure 1. As a direct consequence, we obtain for each compact set $K \Subset \Omega$, positive constants $C$ and $\tilde \rho$,
depending on on $n,p,\Lambda,\lambda, \alpha,\|u\|_{W^{1,p}}, \|\omega\|_{C^{0,\sigma}}, \|f\|_{L^q}$ and $dist(K,\partial \Omega)$, such that for each $x_0\in K$, there holds
\begin{equation}\label{estfinal}
\sup\limits_{B_{\rho}(x_0)}|u(x)-u(x_0)-\nabla u(x_0)\cdot (x-x_0)| \leq  C\rho^{\alpha+1}, \quad 0<\rho\leq \tilde\rho
\end{equation}
where $\alpha$ is determined by (\ref{opt-gamma}). Now, we show how Theorem \ref{thm1} follows from estimate \eqref{estfinal} by a standard argument.

\begin{proof}[Proof of Theorem \ref{thm1}]
Without loss of generality we assume
$x=-re_1,y=r e_1$ with $x,y\in B_{1/4}$, where $e_1=(1,0...,0)$ and $e_i$ is understood similarly. From estimate \eqref{estfinal} we clearly have
$$
u(y)=u(x)+2\partial_1 u(x)r+O(r^{1+\alpha})
$$
and
$$
u(x)=u(y)-2\partial_1 u(y)r+O(r^{1+\alpha}).
$$
The two equations above lead to
$$
|\partial_1 u(x)-\partial_1 u(y)|\cdot r^{\alpha}\le C.
$$
For $i=2,...,n$ we fix $z=\frac{x+y}2+re_i$.
Also by estimate \eqref{estfinal} we have
\begin{equation}\label{f-e1}
u(z)=u(x)+\nabla u(x)\cdot (z-x)+O(r^{1+\alpha})
\end{equation}
and
\begin{equation}\label{f-e2}
u(z)=u(y)+\nabla u(y)\cdot (z-y)+O(r^{1+\alpha}).
\end{equation}
By the definition of $z$ we have
$$
\nabla u(y)\cdot (z-y)=\partial_1 u(y)(-r)+\partial_i u(y)r, \quad \nabla u(x)\cdot (z-x)=\partial_1 u(x)r+\partial_i u(x)r.
$$
Using these equations, from (\ref{f-e1}) and (\ref{f-e2}), we have
$$
u(y)-u(x)-\partial_1 u(y)r+(\partial_i u(y)-\partial_i u(x))r-2\partial_1 u(x)r=O(r^{1+\alpha}).
$$
Using
$$
u(y)-u(x)=2\partial_1 u(x)r+O(r^{1+\alpha})
$$
we have
$$
\partial_i u(y)-\partial_i u(x)=O(r^{\alpha}),\quad i=2,..,n.
$$
Therefore Theorem \ref{thm1} is established by using the estimate above on $K \Subset \Omega$ and a standard covering argument.
\end{proof}

\section{Appendix: Optimal regularity estimates for quasilinear equations with linear growth}\label{ellipsec}

In this section we establish Theorem \ref{thm_2}, which provides optimal regularity estimates for equations \eqref{prop10} with linear growth, i.e., the vector field $a:\Omega \times \mathbb{R}^n \to \mathbb{R}$ satisfying the conditions \eqref{a-cond} and \eqref{rhs} for $p=2$. We remember Theorem \ref{thm_2} performs as a major tool to approach the optimal regularity estimates for small radii, Section \ref{sec3}.

\begin{lem} \label{tool-lem} For $R>0$ and $x_0\in \Omega$, let $h \in H^1(B_R(x_0))$ be a solution of
\begin{equation}\label{homoeq}
-div \, a(x_0,\nabla h) = 0 \quad \mbox{in} \quad B_R(x_0),
\end{equation}
with $a$ satisfying the conditions in \eqref{a-cond} and \eqref{rhs} for $p=2$. Then
\begin{equation}\nonumber
\displaystyle\int_{B_r(x_0)}|\nabla h - (\nabla h)_{x_0,r}|^2 dx \leq C(\lambda,\Lambda)\left(\frac{r}{R} \right)^{n+2} \int_{B_R(x_0)} |\nabla h - (\nabla h)_{x_0,R}|^2 dx,
\end{equation}
for any $0<r\leq R$ and some $C(\lambda,\Lambda)>0$.
\end{lem}
Here we are using the classical average notation
$$
(f)_{x,r}:= \frac{1}{|B_r(x)|}\int_{B_r(x)}f \, dx.
$$

\noindent{\bf Proof:} The proof of Lemma \ref{tool-lem} follows from standard estimates for harmonic functions. First we observe that under the assumptions for $a$ (with $p=2$), $h \in W^{2,2}_{loc}(B_R(x_0))$. Differentiating (\ref{homoeq}) with respect to $x_l$ $(l=1,...,n)$, we have
\begin{equation}
-div (A(x)\nabla h_l)=0, \quad \mbox{in }\quad B_R(x_0)
\end{equation}
where $A_{ij}(x)=\partial_{\xi_j}a^i(x_0,\nabla h(x))$ is uniformly elliptic (because $p=2$):
$$
\lambda |\xi|^2 \leq A_{ij}(x)\xi_i\xi_j \leq \Lambda |\xi|^2, \quad \forall \xi=(x_1,...,x_n)\in \mathbb R^n,
$$
and $C^{0}(\overline{B_R(x_0)})$ because $h$ is $C^1$ as a solution to the constant coefficient equation (\ref{homoeq}).
Lemma \ref{tool-lem} follows immediately from Lemma 1.41 of \cite{HL}, since the coefficient matrix $A$ for $h_l$ satisfies all the requirements of Lemma 1.41 in
\cite{HL}.

\medskip

As an immediate consequence of Lemma\ref{tool-lem}, we have

\begin{lem}\label{lem_2} Let $h \in H^1(B_R(x_0))$ be a solution to \eqref{homoeq} in $B_R(x_0)$, and $u$ be a solution of (\ref{prop10}) in the same domain. There exists $C>0$ depending only on $\lambda,\Lambda$ such that
\begin{equation}\nonumber
\begin{array}{rcl}
\displaystyle\int_{B_r(x_0)}|\nabla u - (\nabla u)_{x_0,r}|^2 dx & \leq & \displaystyle C\left(\frac{r}{R} \right)^{n+2} \int_{B_R(x_0)} |\nabla u - (\nabla u)_{x_0,R}|^2 dx \\
 & + & C\,\displaystyle\int_{B_R(x_0)} |\nabla u- \nabla h|^2 dx\\
\end{array}
\end{equation}
for any $u \in H^1(B_R(x_0))$ and $0<r\leq R$.
\end{lem}

\begin{proof}
Let us consider $v:=u-h$. A direct computation gives
\begin{equation}\nonumber
\begin{array}{ccc}
&&\displaystyle\int_{B_r(x_0)}|\nabla u-(\nabla u)_{x_0,r}|^2 dx\\
 & \leq & C \left(\displaystyle\int_{B_r(x_0)}|\nabla u - (\nabla h)_{x_0,r}|^2 dx+\displaystyle\int_{B_r(x_0)}|\nabla v|^2 dx \right) \\
 & \leq & C \left(\displaystyle\int_{B_r(x_0)}|\nabla h - (\nabla h)_{x_0,r}|^2 dx+\displaystyle\int_{B_r(x_0)}|\nabla v|^2 dx \right)
\end{array}
\end{equation}
for any $0<r\leq R$. In virtue of Lemma \ref{tool-lem} we have
\begin{align*}
&\int_{B_r(x_0)}|\nabla u-(\nabla u)_{x_0,r}|^2 dx \\
\le & \,C \left(\frac{r}{R}\right)^{n+2}\int_{B_R(x_0)}|\nabla h - (\nabla h)_{x_0,R}|^2 dx
+ C \int_{B_R(x_0)}|\nabla v|^2 dx.
\end{align*}
Using triangle inequalities we further have
\begin{align*}
&\int_{B_r(x_0)}|\nabla u-(\nabla u)_{x_0,r}|^2 dx \\
\le &\, C \left(\frac{r}{R}\right)^{n+2}\int_{B_R(x_0)}|\nabla u - (\nabla u)_{x_0,R}|^2 dx+ C \int_{B_R(x_0)}|\nabla v|^2 dx\\[0.15cm]
&+CR^n |\nabla u_{x_0,R}-\nabla h_{x_0,R}|^2\\[0.15cm]
\le &\, C\left(\frac{r}{R}\right)^{n+2}\int_{B_R(x_0)}|\nabla u - (\nabla u)_{x_0,R}|^2 dx+ C \int_{B_R(x_0)}|\nabla v|^2 dx
\end{align*}
where the last inequality can be derived by the definition of $\nabla u_{x_0,R},\nabla h_{x_0,R}$.
\end{proof}

In order to prove Theorem \ref{thm_2}, we need the following technical lemma in \cite[Lemma 3.4]{HL}.

\begin{lem}\label{2lem_2}
Let $\Phi \geq 0$ be a nondecreasing function on $[0,R]$ satisfying
$$
\Phi(\rho) \leq A \left(\left( \frac{\rho}{r} \right)^\tau+\epsilon \right) \Phi(r) + Br^\varsigma
$$
for any $0<\rho\leq r \leq R$ with $A,B,\tau,\varsigma$ nonnegative constants and $\tau>\varsigma$. Then for any $\theta \in (\varsigma,\tau)$ there exists
a constant $\epsilon_0=\epsilon_0(A,\tau,\varsigma,\theta)$ such that if $\epsilon<\epsilon_0$ we have, for all $0<\rho\le r\le R$
$$
\Phi(\rho) \leq c \bigg ( (\frac{\rho}{r})^\theta \Phi(r) + Br^\varsigma \bigg )
$$
where $c$ is a positive constant depending on $A,\tau,\varsigma,\theta$. In particular we have for any $0<r\le R$
$$\Phi(r)\le c \bigg (\frac{\Phi(R)}{R^{\theta}}r^{\theta}+Br^{\varsigma}\bigg ). $$
\end{lem}

\begin{proof}[Proof of Theorem \ref{thm_2}]

First for fixed $R$ and $x_0$ we write the equation for $u$ in $B_R(x_0)$ in the weak form:

$$\int_{B_R(x_0)}a(x_0,\nabla u)\nabla \phi=\int_{B_R(x_0)}f\phi+\int_{B_R(x_0)}(a(x,\nabla u)-a(x_0,\nabla u))\nabla\phi, $$
for all $\phi\in H_0^1(B_R(x_0))$.
For the last term on the right hand side we use the H\"older assumption of $a$ in (\ref{a-cond}):
\begin{equation}\label{wek-u2}
\left|\int_{B_R(x_0)}(a(x,\nabla u)-a(x_0,\nabla u))\nabla v \right|\le C\int_{B_R(x_0)}\omega(|x-x_0|)|\nabla u| |\nabla v|.
\end{equation}
Let $h$ be the unique solution to \eqref{homoeq} that makes $u-h \in H^1_0(B_R(x_0))$.
Denote $v:=u-h$ as a test function to write the equation for $u$ as
\begin{equation}\label{wek-u}
\int_{B_r(x_0)}(a(x_0,\nabla u)\nabla v-fv)=E
\end{equation}
where $E=\displaystyle\int_{B_r(x_0)}(a(x,\nabla u)-a(x_0,\nabla u))\nabla v$ satisfies (by (\ref{wek-u2}))
\begin{equation}\label{e-e}
 |E|\le C\int_{B_r(x_0)}\omega(|x-x_0|)|\nabla u| |\nabla v|.
 \end{equation}
The equation for $h$ is
\begin{equation}\label{eq-h}
\int_{B_r(x_0)} a(x_0,\nabla h)\nabla v=0.
\end{equation}
In order to estimate the difference between (\ref{wek-u}) and (\ref{eq-h}) we first observe that
$$a(x_0,\nabla u)-a(x_0,\nabla h)=\int_0^1\partial_{\xi}a(x_0,t\nabla u+(1-t)\nabla h)dt\cdot (\nabla u-\nabla h). $$
Then the ellipticity assumption on $a$ with $p=2$ gives
\begin{equation}\label{ellip-v}
 (a(x_0,\nabla u)-a(x_0,\nabla h))\cdot \nabla v\ge \lambda |\nabla v|^2.
 \end{equation}
Thus the combination of (\ref{wek-u}), (\ref{e-e}), (\ref{eq-h}) and (\ref{ellip-v}) gives
\begin{equation}\label{1est_2}
\displaystyle\int_{B_r(x_0)}|\nabla v|^2 dx \leq C \left(\omega(r)^2 \int_{B_r(x_0)}|\nabla u|^2 dx+ \left(\int_{B_r(x_0)}|f|^{\frac{2n}{n+2}}dx\right)^{\frac{n+2}{n}}\right)
\end{equation}
for some universal constant $C>0$. Standard H\"older inequality yields
\begin{equation}\label{2est_2}
\left(\int_{B_r(x_0)}|f|^{\frac{2n}{n+2}}dx\right)^{\frac{n+2}{n}} \leq  C(n)\left(\int_{B_r(x_0)}|f|^{q}dx\right)^{\frac{2}{q}}\cdot r^{\,n+2\left(1-n/q\right)}
\end{equation}
for all $q>n$. Then, by \eqref{1est_2} and \eqref{2est_2} we get
\begin{equation}\label{3est_2}
\displaystyle\int_{B_r(x_0)}|\nabla v|^2 dx \leq C \left(\omega(r)^2 \int_{B_r(x_0)}|\nabla u|^2 dx+ \|f\|^2_{L^q(B_1)}\cdot r^{\,n+2\left(1-n/q\right)}\right).
\end{equation}
It follows from Lemma \ref{lem_2} that
\begin{equation}\nonumber
\begin{array}{rcl}
\displaystyle\int_{B_r(x_0)}|\nabla u - (\nabla u)_{x_0,r}|^2 dx & \leq & C\left(\dfrac{r}{R} \right)^{n+2} \displaystyle\int_{B_R(x_0)} |\nabla u - (\nabla u)_{x_0,R}|^2 dx\\[0.5cm]
& + & C\omega(r)^2\displaystyle\int_{B_r(x_0)}|\nabla u-(\nabla u)_{x_0,R}|^2\\[0.7cm]
& + & C(\nabla u_{x_0,R})^2r^{n+2\sigma}\\[0.5cm]
& + & C\|f\|^2_{L^q(B_1)}\cdot r^{n+2(1-n/q)}
\end{array}
\end{equation}
where $\omega(r)=O(r^{\sigma})$ is used.
Thus
\begin{align}
\int_{B_r(x_0)}|\nabla u - (\nabla u)_{x_0,r}|^2 dx & \leq   C\omega(r)^2 \int_{B_R(x_0)} |\nabla u - (\nabla u)_{x_0,R}|^2 dx \nonumber\\[0.15cm]
 & + C\left(\frac{r}{R}\right)^{n+2}\int_{B_R(x_0)} |\nabla u - (\nabla u)_{x_0,R}|^2 dx \nonumber \\[0.15cm]
 & + Cr^{\,n+2\beta}  \label{add-1}
\end{align}
for any $0<r\leq R$ and $\beta=\min\{\sigma,1-n/q\}$ . By Lemma \ref{2lem_2} there is $R_0>0$ such that
\begin{equation}\nonumber
\begin{array}{rcl}
\displaystyle\int_{B_r(x_0)}|\nabla u - (\nabla u)_{x_0,r}|^2 dx & \leq & \displaystyle C \left(\frac{r}{R} \right)^{\,n+2\beta} \int_{B_R(x_0)} |\nabla u - (\nabla u)_{x_0,R}|^2 dx \\[0.5cm]
 & + & C \,r^{\,n+2\beta}\\
\end{array}
\end{equation}
for any $0<r\leq R\le R_0$. In particular, for $R=R_0$ and $0<r \leq R_0$ we have
\begin{equation}\nonumber
\displaystyle\int_{B_r(x_0)}|\nabla u - (\nabla u)_{x_0,r}|^2 dx \leq C \cdot r^{\,n+2\beta}.
\end{equation}
A standard application of Campanato's embedding Theorem (see for instance \cite{MZ}) proves the desired H\"older continuity. The proof of Theorem \ref{thm_2} is complete.
\end{proof}

\bigskip

{\small \noindent{\bf Acknowledgments.} The authors would like to thank the hospitality of the University of Florida, where this work was conducted. DJA is supported by CNPq-Brazil. LZ is partially supported by a Simons Foundation Collaboration Grant.}


\end{document}